\documentclass[12pt]{amsart}
\usepackage{amsfonts, amsthm, amssymb, amsmath, stmaryrd}
\usepackage{mathrsfs,array}
\usepackage{eucal,color,times,enumerate,accents}
\usepackage[all]{xy}
\usepackage{url}
\usepackage{enumitem}
\usepackage[pdftex,bookmarksnumbered,bookmarksopen]{hyperref}
\hypersetup{
  colorlinks,
  citecolor=black,
  linkcolor=black,
  urlcolor=black}
  
  \newcommand{\MT}{\mathbf{MT}}

\usepackage{tikz-cd}
\usepackage{leftidx}
\usetikzlibrary{positioning}
\usepackage{dsfont}
\usepackage{tikz}
\usetikzlibrary{matrix,arrows,decorations.pathmorphing,positioning}

\newcommand{\HL}{\textnormal{HL}}
\newcommand{\an}{\textnormal{an}}

\usepackage[utf8]{inputenc}
\usepackage[T1]{fontenc}

  \newcommand{\Addresses}{{
  \bigskip
  \footnotesize

\textsc{I.H.E.S., Universit\'e Paris-Saclay, CNRS, Laboratoire Alexandre
  Grothendieck. 35 Route de Chartres, 91440 Bures-sur-Yvette
  (France)}\par\nopagebreak
  \textit{E-mail address}, G.~Baldi: \texttt{baldi@ihes.fr}, \texttt{baldi@imj-prg.fr}  

  \medskip

  \textsc{Humboldt Universit\"{a}t zu Berlin (Germany)}
  \par\nopagebreak
  \textit{E-mail address}, B.~Klingler: \texttt{bruno.klingler@hu-berlin.de}
  
  \medskip
  
\textsc{I.H.E.S., Universit\'e Paris-Saclay, CNRS, Laboratoire Alexandre Grothendieck. 35 Route de Chartres, 91440 Bures-sur-Yvette (France)}\par\nopagebreak
  \textit{E-mail address}, E.~Ullmo: \texttt{ullmo@ihes.fr}
}}

\usepackage[nameinlink]{cleveref}

\usepackage{bbm}
\newcommand{\VV}{\mathbb{V}}

\newcommand{\ad}{\textnormal{ad}}

\usepackage{leftidx}

\usepackage[bottom=3cm, top=3cm, left=3cm, right=3cm]{geometry}

\input xy
\xyoption{all}

\newcommand{\PP}{\mathbb{P}}
\usepackage{tikz-cd}
\theoremstyle{plain}
\newtheorem{thm}{Theorem}[]

\newtheorem{conj}{Conjecture}[]

\newcommand{\G}{{\mathbf G}}

\newcommand{\atyp}{\textnormal{atyp}}

\newcommand{\typ}{\textnormal{typ}}
\newcommand{\pos}{\textnormal{pos}}

\newtheorem{prop}[thm]{Proposition}
\newtheorem{cor}[thm]{Corollary}
\theoremstyle{definition}
\newtheorem{defi}[thm]{Definition}

\newtheorem{rmk}[thm]{Remark}
\newtheorem{question}[thm]{Question}
\theoremstyle{remark}

\numberwithin{equation}{section}

\newcommand{\CC}{\mathbb{C}}

\DeclareMathOperator{\NL}{NL}

\DeclareMathOperator{\codim}{codim}

\newcommand{\ev}{\operatorname{ev}}

\newcommand{\SO}{\operatorname{SO}}

\newcommand{\Z}{\mathbb{Z}}
\newcommand{\Q}{\mathbb{Q}}

\newcommand{\R}{\mathbb{R}}

\newcommand{\Oo}{\mathcal{O}}

\newcommand{\C}{\mathbb{C}}

\newcommand{\prim}{\textnormal{prim}}

\makeatletter
\def\@settitle{\begin{center}%
  \baselineskip14\p@\relax
  \bfseries
  \uppercasenonmath\@title
  \@title
  \ifx\@subtitle\@empty\else
     \\[1ex]\uppercasenonmath\@subtitle
     \footnotesize\mdseries\@subtitle
  \fi
  \end{center}%
}
\def\subtitle#1{\gdef\@subtitle{#1}}
\def\@subtitle{}
\makeatother

\usepackage{microtype}
\usepackage{fullpage}

\begin{document}

\newcommand{\adjunction}[4]{\xymatrix@1{#1{\ } \ar@<-0.3ex>[r]_{ {\scriptstyle #2}} & {\ } #3 \ar@<-0.3ex>[l]_{ {\scriptstyle #4}}}}

\title{Non-density of the exceptional components of the Noether-Lefschetz locus}
\date{\today}

\author{Gregorio Baldi, Bruno Klingler, and Emmanuel Ullmo}

\begin{abstract}
We study when the Picard group of smooth surfaces of degree
$d\geq 5$ in $\mathbb{P}^3$ acquires extra classes. In particular we
show that the so called \emph{exceptional components} of the
Noether-Lefschetz locus are not Zariski dense. This answers a 1991
question of C. Voisin. We also obtain similar results for the
Noether-Lefschetz locus for suitable $(Y,L)$, where $Y$ is a smooth
projective threefold and $L$ a very ample line bundle. Both results
are applications of the Zilber-Pink viewpoint recently developed by
the authors for arbitrary (polarized, integral) variations of Hodge
structures. 
\end{abstract}

\maketitle
\tableofcontents

\section{Introduction}
We work over $\C$ and fix an integer $d \geq 4$. Let
$U_d=\mathbb{P}H^0(\PP^3, \Oo(d))-\Delta$ be the scheme parametrizing
smooth surfaces $X$ of degree $d$ in $\PP^3$. Consider the so called
\emph{Noether-Lefschetz locus}: 
\begin{displaymath}
\NL_d:=\{[X]\in U_d : \operatorname{Pic}(\PP^3)\to
\operatorname{Pic}(X) \text{  is not an isomorphism} \}. 
\end{displaymath}
Each $X$ outside $\NL_d$ has the following pleasant and useful
property: every curve on $X$ is the complete intersection of $X$ with
another surface in $\mathbb{P}^3$.

The locus $\NL_d$ has been the subject of several investigations. In
1880 Noether stated that $\NL_d$ is a countable union of strict
irreducible algebraic subsets of $U_d$. This was 
proved in the 1920s by Lefschetz, and the topic flourished in the 1980s, in
large part due to the works of Griffiths and his
school. We refer to the notes \cite{zbMATH06342071} for a
discussion.

\begin{rmk}
  Each irreducible component of $\NL_d$ has a natural schematic
  structure, which is often non-reduced. In what follows we will always
  consider these components with their reduced structure.
  \end{rmk}

In this paper we study the relationship between the following theorem
and the recent general results on the distribution of the Hodge
locus that we obtained in \cite{2021arXiv210708838B}:
\begin{thm}[Explicit Noether-Lefschetz theorem (Green-Voisin)]\label{explicitNL} 
  Each irreducible component $Y$ of $\NL_d$ is
  such that 
\begin{displaymath}
d-3 \leq \codim_{U_d}Y \leq h^{2,0}= \ {{d-1}\choose{3}}.
\end{displaymath}
\end{thm}

The upper bound on codimension is elementary: a class $\int_{\lambda} \in
H^{2}(X,\Q)$ is algebraic if and only if it has type (1,1), which is equivalent
to $\int_{\lambda } \omega =0$ for each $\omega \in
H^0(X,\Omega^2_X)$). It is also easy to show that the family of smooth
degree~$d$ surfaces containing a line forms a component of $\NL_d$ of
codimension $d-3$, see
e.g. \cite[Example 1.1]{zbMATH06342071}. On the other hand, showing that
this provides a lower bound is more subtle and depends on
fairly delicate algebraic considerations. See
\cite{zbMATH03891508, zbMATH04103223, zbMATH03997985}. 

\medskip
In view of \Cref{explicitNL} the following definition is natural, it
appeared explicitly\footnote{In the \cite{zbMATH04097545} the terminology \emph{special} is used in place of \emph{exceptional}. In order to avoid a possible source of confusion, since the word `special' is at odds with its use in \cite{2021arXiv210708838B}, we decided rename it from the start.} in \cite[(page 668)]{zbMATH04097545}; see also the previous
works \cite{zbMATH03842038, zbMATH03879066}: 
\begin{defi}\label{defityp}
A component $Y $ of $\NL_d$ is said to be \emph{general} if it has
codimension $h^{2,0}$, and \emph{exceptional} otherwise. 
\end{defi}

It is known that the union of the general components of $\NL_d$ is Zariski-dense in $U_d$
\cite{zbMATH04097545}, and that $\NL_d$ is even dense in $U_d$ for the
analytic topology (an argument of Green, written for instance in \cite[Prop. 5.20]{zbMATH01927232}). See also
\cite{zbMATH00060091} for the construction of some explicit exceptional
components. We refer also to the recent works
\cite{2022arXiv221110592E, 2023arXiv230316179K}  (inspired by the
Zilber-Pink paradigm proposed in \cite{2021arXiv210708838B}). In particular such
results show, abstractly, the existence of general
components in $\NL_d$ and the density of their union. See also \Cref{remmm} and (1) in
\Cref{thmfinal}, for a more precise statement and detailed discussion
on this important point.

\subsection{Harris conjecture and a question of Voisin}
From now on, we assume $d \geq 5$ (quadrics and  cubic surfaces
are both rational and will have no non-trivial periods; while for $d=4$, there are
no exceptional components). Harris conjectured in
\cite[(page 301)]{zbMATH04103221} that 
$\NL_d$ has only finitely many exceptional components. This conjecture was first
disproved by Voisin in \cite{zbMATH00027662}, for $d$ big
enough and divisible by 4. We refer also to \cite[Example
3.10]{zbMATH06342071} for a brief overview of the 
ingenious construction of Voisin, which ultimately comes from
$\NL_4$. 

\begin{rmk}
If one takes into account the natural more complicated schematic
structure on each component of $\NL_d$, Dan showed recently in
\cite{zbMATH07414659} that for $d\geq 
6$ there exists infinitely many such components with the same underlying
reduced component, thus disproving the stronger Harris conjecture in this
setting. More precisely, he shows that the space parametrizing smooth, degree $d$ 
surfaces containing 2 coplanar lines can be equipped with
infinitely many (distinct) scheme structures naturally arising as the
Hodge loci associated to different combinations of the two coplanar
lines (see Theorem 2.4 in \emph{op. cit.}).
\end{rmk}

After disproving Harris conjecture, Voisin asked the following
\cite[0.5]{zbMATH00027662} (see also \cite[Question
3.11]{zbMATH06342071}):  
\begin{question}(Voisin)\label{questionv}
Is the union of exceptional components of the Noether-Lefschetz locus Zariski dense in $U_{d}$?
\end{question}

The main result of this note provides a negative answer to
\Cref{questionv} (as expected by Voisin):

\begin{thm}\label{mainthm0}
For $d \geq 5$, the union of the exceptional components of $\NL_d$ is
not Zariski dense in $U_d$.
\end{thm}

We also briefly discuss a related question on the generic Picard rank
of the components of $\NL_d$.
\begin{question}(Ciliberto-Harris-Miranda, \cite[(page
  668)]{zbMATH04097545}) \label{rank}
What is the Picard group of a surface corresponding to a general point
of a general component of $\NL_d$?
\end{question}
The answer is expected to be $\Z^2$, but this is not known (see also \cite{zbMATH00007580} for
related discussion and results). Let us call a component $Y$ of $\NL_d$
\emph{Picard generic} if the Picard group of a surface corresponding
to a general point of $Y$ is $\Z^2$, and \emph{Picard exceptional}
otherwise. The same arguments behind the proof of \Cref{mainthm0} give
also the following:

\begin{thm}\label{lastthm}
The union of the Picard generic general components of $\NL_d$ is dense
in $U_d(\C)$, while the union of the 
Picard exceptional components is contained in a finite union of strict
\emph{special} 
subvarieties of $U_{d}$ for $\VV$ (in particular it is not 
Zariski-dense in $U_d$). 
\end{thm}

In \Cref{sec4} we discuss generalizations of the theory
for very ample surfaces in arbitrary smooth threefolds, see
\Cref{thmfinal}. There we obtain \emph{abstract} finiteness
results for the exceptional components, and existence and density of the
general ones. Once more, this is a special case of the \emph{completed
  Zilber-Pink paradigm} envisioned by the authors in
\cite{2021arXiv210708838B}. (The adjective \emph{completed} refers to
the fact that we look at both the atypical and typical intersections,
whereas, in the original ZP viewpoint, only atypical intersections are
considered). See also \Cref{conj0}, for a similar finiteness regarding
a Lefschetz pencil $\mathcal{Y}\to \mathbb{P}^1$ that would follow
from the zero-dimensional case of ZP. 

\subsection{Remarks and questions}

Little is known about the exceptional components of $\NL_d$ (see 
\Cref{rmkhigher} for a brief discussion of the higher dimensional
case). Voisin described the components of maximal dimension:
\begin{thm}[Voisin]\label{thm2}
The components of $\NL_d$ have codimension $> 2d-7$, with the exception of the family of surfaces containing:
\begin{enumerate}
\item A line (this one has codimension $d-3$);
\item A conic (codimension $2d - 7$).
\end{enumerate} 
\end{thm}
The proof can be found in \cite{zbMATH04193894}. See also
\cite{zbMATH04103221, zbMATH00005091} for related results. In particular the Harris conjecture
is true for $d=5, 6,7$. We refer to \cite{zbMATH00562487} for a discussion about other
interesting finiteness related to the Noether-Lefschetz locus and to
\cite{zbMATH07383960} for a conjectural description of an infinite
number of (reduced) exceptional components of $\NL_d$. 

One can ask to what extent Voisin's construction is the only source
of counterexamples to Green's conjecture. For brevity we just discuss
the case where $d=p$ is a prime number--a similar question can be
formulated in general by defining \emph{primitive surfaces} as
the ones that do not come from a pull-back along $F: \mathbb{P}^3\to
\mathbb{P}^3$, where $F$ is given by polynomials of degree $d'$ with
no common zeroes, of surfaces of degree $d/d'$: 
 
\begin{question}[Refined Harris conjecture]\label{seconque}
For $p$ a prime number, are there only finitely many exceptional components of $\NL_p$?
\end{question}
As explained in the previous section, \Cref{mainthm0} cannot in general
 be improved to a finiteness statement but, in the case where $d$ is a
prime number, as far as we are aware, the finiteness proposed by
Harris could still be true. Finally we refer to \Cref{cordynk} for a simple example of how \Cref{mainthm0} can lead to finiteness results.

\subsection{Acknowledgements}
The first author thanks the organizers of the conference \emph{On
  Noether-Lefschetz and Hodge loci} (Aug. 2023) held at IMPA (in Rio
de Janeiro), where he attended several inspiring talks related to
$\NL_d$, by H. Movasati, R. Villaflor, and J. Duque. He also thanks
D. Urbanik and N. Khelifa for related discussions. B.K. is partially
supported by the Grant ERC-2020-ADG n.101020009 - TameHodge. 
Finally we thanks the referees for their careful reading.

\section{Some cases of Zilber-Pink}\label{section2}
This section is a specialization of the results and vocabulary
introduced in \cite{2021arXiv210708838B} in a much greater
generality. To understand the (polarized) variation of Hodge structures (VHS) $\VV \to S=U_d$ interpolating the
primitive cohomology groups $H^2(X,\Z)_{\prim}$, we look at the
associated holomorphic period map 
\begin{equation} \label{period0}
  \Phi: 
  S^{\an} \to \Gamma \backslash D,
\end{equation}
completely describing $\VV$.
Here $S^\an$ denotes the complex manifold analytification of $S$, $(\G, D)$ denotes the generic Hodge datum of $\VV$
and 
$\Gamma \backslash D$ is the associated Hodge variety. As famously
proved by Lefschetz \cite{zbMATH02597379} (see also \cite[Section
1.2]{zbMATH06342071}, and the discussion in
\cite[Rmk. 9.1]{2021arXiv210708838B}), we have that
$\G^{\operatorname{ad}}$ is $\SO(2h^{2,0},h^{1,1}-1)$, where $h^{2,0}=
\dim_\C H^{2,0}(X)$ and $h^{1,1}=\dim_{\C} H^{1,1}(X)$, for some
$[X]\in U_d$ (from now on we write $h^{1,1}_{\prim} :=h^{1,1}
-1$)\footnote{Implicitly we use the André-Deligne monodromy theorem, which rests upon the fixed-part and the semisimplicity
theorems, to observe that the monodromy is a normal subgroup of
  the derived subgroup of the generic Mumford-Tate group.}. The
Noether-Lefschetz locus lies in the more general tensorial Hodge locus
of $\VV_s$, namely the set of closed points $s\in S$ where the
Mumford-Tate group of the Hodge structure $\VV_s$ drops: 
\begin{displaymath}
\HL (S,\VV^{\otimes}):=\{s\in S(\C) : \MT(\VV_s)\subsetneq \MT(\VV)\}.
\end{displaymath}
In particular:
\begin{equation}\label{eqqq}
\NL_d\subset \HL(S,\VV^{\otimes})_{\pos}.
\end{equation}

The fact that we can consider just the Hodge locus $\HL(S,\VV^{\otimes})_{\pos}$ of \emph{positive
  period dimension} comes from the bound recalled in \Cref{explicitNL}
and the local/infinitesimal Torelli theorem of Griffiths \cite[Theorem
9.8,
Part I]{zbMATH03341218} which asserts that the differential of the period
map $ \Phi$ is injective at any point $s\in S$ (modulo the action of
the projective group $\mathbf{PGL}(3, \CC)$ on $U_d$), as long as $d\geq
4$.

Following \cite{2021arXiv210708838B}, an irreducible
subvariety $Z$ of $S$ is called \emph{special}
(resp. \emph{weakly special}) if any irreducible subvariety $W\subset S$ strictly
containing $Z$ has generic Mumford-Tate (resp. algebraic monodromy)
strictly bigger than the one of $Z$. 

The next definition of \cite{2021arXiv210708838B} distinguishes special subvarieties between the ones having
the ``expected'' codimension and the ones that
don't:
\begin{defi}\label{atypical}
  Let $Z\subset S$ be a special
  subvariety for $\VV$, with generic Hodge datum $(\G', D')$. It is
  said to be \emph{atypical} if $\Phi(S^{\an})$ and $\Gamma'
  \backslash D'$ do not intersect generically along $\Phi(Z^{\an})$: 
  \begin{equation} \label{equation atypical}
    \codim_{\Gamma\backslash D} \Phi(Z^{\an}) < \codim_{\Gamma\backslash
    D} \Phi(S^{\an}) + \codim_{\Gamma\backslash D} \Gamma'\backslash
  D'\;\;,
  \end{equation}
 
  \noindent
  Otherwise it is said to be \emph{typical}.  The \emph{atypical Hodge locus} $\HL(S,\VV^\otimes)_{\atyp} \subset \HL(S, \VV^\otimes)$
  (resp. the \emph{typical Hodge locus} $\HL(S,\VV^\otimes)_{\typ} \subset \HL(S, \VV^\otimes)$) is
  the union of the atypical (resp. strict typical) special
  subvarieties of $S$ for $\VV$. 
\end{defi}

\begin{rmk}
The above definition makes implicitly use of the
algebraicity theorem of the Hodge locus for arbitrary Hodge classes
\cite{CDK}: Lefschetz (1,1) theorem will not suffices for our
purposes. 
\end{rmk}

\subsection{Noether-Lefschetz components as special varieties}\label{sectionNL}
The upper bound $h^{2,0}$ from \Cref{explicitNL} can also be seen as
follows (see for example \cite[(page 46)]{carlson}, for more details). Let $\mathfrak{g}$ be the Hodge-Lie algebra associated to
$(\mathbf{G},D_G)$. We have  
\begin{displaymath}
2\dim \mathfrak{g}^{-2,2}=(h^{2,0}-1)h^{2,0}\;,
\end{displaymath}
\begin{displaymath}
\dim \mathfrak{g}^{-1,1}=h^{2,0} \, h^{1,1}_{\prim}.
\end{displaymath}
The condition defining $\NL_d$ is the one of having an extra Hodge
vector in $\VV$. In particular an irreducible component of $\NL_d$ is
a special subvariety of $U_d$ for $\VV$ of the form $\Phi^{-1}(\Gamma_H \backslash D_H)^0$, where
$\mathbf{H}\subset \mathbf{G}$ is isomorphic, over the real numbers, to the group
$\SO(2h^{2,0},h^{1,1}_{\prim}-1)$; and $(\cdot)^0$ denotes an
irreducible component.

At the level of the Hodge-Lie algebra, we observe that
$\mathfrak{h}^{-2,2}=\mathfrak{g}^{-2,2}$ and that $\dim
\mathfrak{h}^{-1,1}=h^{2,0}(h^{1,1}_{\prim}-1)$, as we are
parametrising the Hodge tuple
$(h^{2,0},(h^{1,1}_{\prim}-1),h^{0,2})$.
The components of $\NL_d$ thus have codimension in $U_d$ at most equal
to
$$ \codim_{\Gamma\backslash D} (\Gamma_H \backslash D_H) =
\codim_{D} D_H = \dim
\mathfrak{g}^{-1,1} - \mathfrak{h}^{-1,1} = h^{2,0}\;.$$

\subsection{Results and conjectures from our earlier work}
From \cite[Section 2]{2021arXiv210708838B}, we expect:

\begin{conj}[Zilber--Pink conjecture] \label{main conj}
 The atypical Hodge locus
  $\HL(S,\VV^\otimes)_{\atyp}$ is a finite union of atypical special
subvarieties of $S$ for $\VV$.
\end{conj}

\begin{conj}[Density of the typical Hodge locus] \label{conj-typical}
 If 
$\HL(S, \VV^\otimes)_\typ$ is non-empty then it is dense (for the analytic topology) in
$S$.
\end{conj}

We recall here a special case of \cite[Thm. 6]{2021arXiv210708838B}
(we refer to \emph{op. cit.} for more precise statements). (Regarding
the latter conjecture, see \cite[Thm. 10.1 and
Rmk. 10.2]{2021arXiv210708838B}, we notice here that it is a consequence of the former). 
\begin{thm}[Geometric Zilber--Pink]\label{geometricZP}
Let $Z$ be an irreducible component of the Zariski closure
of the union of the atypical special subvarieties of positive period
dimension in $S$. Then: 

\begin{itemize}
  \item[(a)] Either $Z$ is a maximal atypical special subvariety;
    \item[(b)] Or the adjoint Mumford-Tate group $\G_Z^\ad$
  decomposes as a non-trivial product $\mathbf{H}^\ad_Z \times \mathbf{L}_Z$; $Z$ contains a Zariski-dense
set of fibers of $\Phi_{\mathbf{L}_{Z}}$ which are atypical weakly special
subvarieties of $S$ for $\Phi$, where (possibly up to an \'{e}tale covering) $$
\Phi_{|Z^\an}= (\Phi_{\mathbf{H}_{Z}}, \Phi_{\mathbf{L}_{Z}}): Z^\an \to  \Gamma_{\G_{Z}}\backslash D_{G_{Z}}= \Gamma_{\mathbf{H}_{Z}}
\backslash D_{H_{Z}} \times  \Gamma_{\mathbf{L}_{Z}}\backslash D_{L_{Z}} \subset \Gamma \backslash
D
\;\;;$$
and $Z$ is Hodge generic in a (Hodge) special subvariety $\Phi^{-1}(
\Gamma_{\G_{Z}}\backslash D_{G_{Z}})^0$ of $S$ for $\Phi$ which is typical. 
\end{itemize}

\end{thm}

\begin{rmk}\label{rmkhigher}
In the terminology of \cite[Section 4.6]{2021arXiv210708838B}, as long
as $d\geq 5$, the VHS $\VV\to U_d$ has level 2. In higher
level, we proved that every irreducible component of the Hodge locus
is atypical. This applies to the parameter space of of smooth
hypersurfaces of degree $d$ in $\mathbb{P}^{n+1}$, if $n=3$ and $d
\geq 5$; $n=4$ and $d\geq 6$; $n=5,6,8$ and $d \geq 4$; and $n=7$ or
$\geq 9$ and $d \geq 3$, see Corollary 1.6 in \emph{op. cit.}. In
particular, in these cases, the Hodge locus of positive period
dimension is itself non-Zariski dense. On a different direction, such
result has also found some number-theoretic application to the
integral points of the moduli spaces of smooth hypersurfaces
\cite{zbMATH07745044}. 
\end{rmk}

\section{Proofs}
Let $d\geq 5$. We can now state the precise version of \Cref{mainthm0}:
\begin{thm}\label{mainthm}
The union of the exceptional (i.e. of codimension $< h^{2,0}$)
components of $\NL_d$ is contained in a finite union of strict \emph{special}
subvarieties of $U_{d}$ for $\VV$ (in particular it is not
Zariski-dense in $U_d$). 
\end{thm}
In particular each irreducible component $W$ of the Zariski closure of
the union of the exceptional components is contained in a strict special
subvariety of $U_d$ of the form $\Phi^{-1}(\Gamma_W
\backslash D_W)^0$, where $(\mathbf{G}_W,D_W)$ is its associated Hodge
datum, strictly contained in the generic Hodge datum $(\mathbf{G},D)$ of
$(U_d,\VV)$. Since the components of $\NL_d$ are maximal among the
components of the Hodge locus of $\VV$ having an extra Hodge vector,
either $W$ is one of the exceptional components, or
the Mumford-Tate group $\mathbf{G}_W$ has to be associated to some
non-trivial Hodge class in some tensorial construction of $\VV_{|
  W}$. We don't know whether such a Hodge tensor comes from an
algebraic cycle in a suitable power of the generic surface $X$ of
$W$. This is what happens in \cite{zbMATH00027662}, and \Cref{mainthm} could be helpful to approach
\Cref{seconque}. Indeed, we record here an application to a finiteness question of the reasoning just explained.

\begin{cor}\label{cordynk}
Among the exceptional components of $\NL_d$, only finitely many have algebraic monodromy isomorphic, over the real numbers, to the group
$H:=\SO(2h^{2,0},h^{1,1}_{\prim}-1)$.
\end{cor}
\begin{proof}[Proof of \Cref{cordynk}]
If there were infinitely many exceptional components with monodromy group (abstractly) isomorphic to $H$, \Cref{mainthm} would imply the existence of a sub-Mumford-Tate datum $(\mathbf{G}_W,D_W)$ strictly contained in the generic Hodge datum $(\mathbf{G},D)$ with the property that $\mathbf{G}_{W,\R}$ strictly contains a subgroup isomorphic to $H$. The very last assertion contradicts the maximality of $\SO(2h^{2,0},h^{1,1}_{\prim}-1)$ in $G=\SO(2h^{2,0},h^{1,1}_{\prim})$, established by Dynkin \cite[Thm. 1.2]{zbMATH03125754} (the argument in \emph{op. cit.} is over $\C$, but the statement over $\R$ follows immediately in this case). Therefore the components of $\NL_d$ under examination form a finite collection.
\end{proof}

The main ingredient in the proof of \Cref{mainthm} is the following:

\begin{prop}\label{mainprop}
Every irreducible component $Y$ of $\NL_d$ is special of positive
period dimension, and:
\begin{enumerate}
\item If $Y$ is exceptional (in the sense of \Cref{defityp}), then it is atypical (in the sense of \Cref{atypical});
\item If $Y$ is general, then it is typical if and only if it it has
  the expected generic Mumford-Tate group  (i.e. isogenous to
  $\mathbb{G}_m\times \operatorname{SO}(2h^{2,0},h^{1,1}_{\prim}-1))$.  
\end{enumerate}
\end{prop}
\begin{rmk}\label{remmm}
In particular a general component of $\NL_d$ with algebraic monodromy/generic
Mumford-Tate group strictly contained in the expected one is atypical.
Once more the union of such components will not be Zariski dense in $U_{d}$ in
virtue of \Cref{geometricZP}. Therefore we observe
that the (analytic) density of $\NL_d$ in $U_d$ comes from the union
of its components which are at the same time general and typical. We are not aware of the existence of
components which are at the same time general and atypical in this particular setting.
\end{rmk}
\begin{proof}
We use the notation from \Cref{sectionNL}. The argument given
there shows that $Y$ is special (of positive period dimension by~(\ref{period0})), of the form 
$\Phi^{-1}(\Gamma_H \backslash D_H)^0$. In particular one expects that 
\begin{displaymath}
\codim _{U_d} Y \leq \codim_D D_H=h^{2,0}
\end{displaymath}
and the equality to be the ``general case''. To understand
\Cref{atypical} we have to normalize the above intersection: indeed it
is possible that the monodromy datum associated to $Y$ is of the form $(\mathbf{H}_0,  D_{H_0})$, for
some $ D_{H_0}$ strictly contained in
$ D_{H}$. Let $(\mathbf{H}_Y, D_Y)$ be the monodromy
datum associated to $Y$ (or, more precisely, to $\VV$ restricted to a
normalization of the reduced scheme associated to $Y$). By
\Cref{atypical}, $Y$ is typical if and only if  
\begin{displaymath}
\codim _{U_d} Y = \codim _D D_Y
\end{displaymath}
(and atypical otherwise).

 By construction $(\mathbf{H}_Y, D_Y)$ is contained in $(\mathbf{H},D_H)$, and therefore
 \begin{displaymath}
 \codim _D D_Y= \codim_D D_H + \codim_{D_H} D_Y.
 \end{displaymath}
Hence $Y$ is typical if and only if
 \begin{displaymath}
 \codim _{U_d} Y =h^{2,0}+\codim_{D_H} D_Y.
 \end{displaymath}
 From the bound in \Cref{explicitNL}, we see that general agrees with
 typical if the monodromy is the expected one, and exceptional implies
 atypical. 
\end{proof}

\begin{proof}[Proof of \Cref{mainthm}]
\Cref{mainthm} follows immediately from
\Cref{mainprop} (1), and \Cref{geometricZP}. Indeed, as remarked in \Cref{section2}, the generic
monodromy group of $(U_d,\VV) $ is simple, and therefore in the case (b)
of the alternative in \Cref{geometricZP} the special subvariety $Z$
has to be strict. 
\end{proof}

\begin{proof}[Proof of \Cref{lastthm}]
Thanks to the previous result, it is enough to analyse the components of $\NL_d$ of codimension equal to $h^{2,0}$. The Picard generic components are typical in the sense of \Cref{section2}, and Picard exceptional exceptional are atypical. Therefore the claim follows, arguing as in the proof of of \Cref{mainthm}.
\end{proof}

To explain the full power of \Cref{main conj} (compared to its
geometric counterpart), we record here the following special
conjectural case, which really involves the zero-dimensional
components and is beyond our current understanding. 

\begin{conj}\label{conj0}
Suppose that $\mathcal{Y}\to \mathbb{P}^1$ is a Lefschetz pencil of
degree $d$ surfaces in $\mathbb{P}^3$. If $d\geq 5$ the Picard number
of $\mathcal{Y}_s$ is greater or equal to $2$ for at most a finite
number of values of $s\in \mathbb{P}^1 (\C)$.  
\end{conj}
The proof of \Cref{mainthm} and the previous arguments show that the
above is a special case of \Cref{main conj}. In fact, this special
case of the Zilber-Pink conjecture was noticed a long time ago by de
Jong \cite[Sec. 3.2.2.]{dj} and should be attributed to him (see also
\cite[Sec. 3.6.]{2021arXiv210708838B} for a history of the Zilber-Pink
conjecture, and a comparison with de Jong's unpublished viewpoint). 

\section{Zilber-Pink and the Noether-Lefschetz locus for arbitrary threefolds}\label{sec4}
To elucidate the role of the Zilber-Pink viewpoint, we conclude with a
general result on the distribution of the Noether-Lefschetz locus for
arbitrary $(Y,L)$ where $Y$ is a smooth projective threefold and $L$ a
very ample line bundle on $Y$. Similar considerations apply also to
the case of quasi-smooth surfaces in a weighted projective space
$\mathbb{P}(q_0,q_1,q_2,q_3)$ (see \cite{zbMATH03799831,
  zbMATH04123897, zbMATH04137950}) as well as surfaces that arise as
complete intersections (see for example \cite{zbMATH06050954}). We
consider $U_{Y,L} \subset \mathbb{P}H^0(L)$, the parameter space of
smooth surfaces $X\subset Y$ such that $\mathcal{O}_Y(X)=L$
(i.e. surfaces in the same equivalence class as $L$), and define 
\begin{displaymath}
\NL_L:=\{[X] \in U_{Y,L} : \operatorname{Pic}(Y)\to \operatorname{Pic}(X) \text{  is not a surjection} \}.
\end{displaymath}
(See \cite{zbMATH05213812, zbMATH06342071} for details about why the connected components of the Picard groups can be ignored).

 Moishezon \cite{zbMATH03260473} found the exact conditions for a
 \emph{Noether-Lefschetz theorem} to hold for $(Y,L)$, namely the
 generic member of $U_{Y,L}$ must have a non-zero evanescent
 $(2,0)$-cohomology class. Our exposition is inspired by the setting
 considered by Lopez and Maclean \cite{zbMATH05213812} and their
 open-ended question (see the last paragraph of page 2 in \emph{op. cit.}):
 
 \begin{question}[Lopez-Maclean]
 How large can the components of $\NL_L$ be in comparison with $U_{Y,L}$?
 \end{question}
 
We see $\NL_L$, in a very Hodge theoretic fashion, as the zero locus
of a vector bundle section. Let $\VV\to U_{Y,L}$ be the polarized
$\Z$-VHS interpolating the $H^2(X,\Z)_{\ev}$, and assume from on that
$h^{2,0}_{\ev}\neq 0$. See \cite[Sec. 2]{zbMATH05213812} for a brief
discussion of vanishing cohomology (we very much follow their
notation and viewpoint, but we changed the terminology from \emph{evanescent} cohomology to vanishing), \cite{zbMATH01927232} for a discussion about
primitive and vanishing cohomologies. 
 
 We simply present the general axiomatic statement, but several
 concrete corollaries can be obtained. As the reader might notice, the
 body of literature on these questions is quite large. We don't attempt to survey the most general statements,
 but limit ourselves to concrete applications of the Zilber-Pink
 paradigm, that we hope will complement several nice ideas and
 viewpoints that emerged from the study of the Noether-Lefschetz locus
 of smooth surfaces. In the hope that the result below will have
 further algebro-geometric applications, we refer also to
 \cite{zbMATH00850005, zbMATH05503406} where various cycle theoretic
 applications related to the knowledge of $\NL_L$ are discussed (among
 which one also finds Koll\'{a}r's counterexamples to the integral
 Hodge conjecture). 
 \begin{thm}\label{thmfinal}
If the following are satisfied:
\begin{itemize}
\item[A1.] The algebraic monodromy group of $\VV\to U_{Y, L}$ contains the full special orthogonal
  stabilizing the intersection form; 
\item[A2.] The infinitesimal Torelli theorem for $(U_{Y,L},\VV)$ holds true;
\item[A3.] $\dim U_{Y,L}>h^{2,0}_{\ev}$ ($>0$, as
  assumed from the beginning); 
\end{itemize} 
then the locus $\NL_L$ consists of countably many irreducible (strict)
algebraic subsets of $U_{Y,L}$, behaving as follows: 
\begin{enumerate}
\item The union of the atypical components of $\NL_L$ (in
  particular the ones of codimension $< h^{2,0}_{\ev}$) is not Zariski
  dense: its Zariski closure is contained in a finite union of strict
  special subvarieties; 
\item The union of the typical components is dense in $U_{Y,L}(\C)$
  (in particular it is non-empty and they have codimension
  $h^{2,0}_{\ev}$). 
\end{enumerate}
 \end{thm}
(If one doesn't wish to assume A2. it is enough to replace
``dimension'' by ``period dimension'' and the result is still true.) 
Regarding special cases the existence of components of $\NL_{Y}$ of
maximal codimension, we refer also to \cite{zbMATH07500494} and
\cite{zbMATH06908314}.  
\begin{rmk}
 Several concrete cases can be obtained by referring to the opportune
 results from the literature. For example, regarding A1. see (the
 proof of) \cite[Thm. 9.1]{zbMATH01425212}, for
 A2. \cite{zbMATH03944005} (as well as \cite{zbMATH03985397}, for the
 weighted projective case). Such results often require to replace $L$
 by $L^d$, for $d$ big enough. We didn't attempt to give the optimal
 results. 
\end{rmk}

\begin{proof}[Proof of \Cref{thmfinal}]
We study $(\VV, U_{Y, L})$ via the period map $\Psi : U_{Y, L}\to
\Gamma \backslash D$ (which we know to be non-trivial, thanks to A1
and A2). Condition A3. asserts that we would expect (non-empty)
components of $\NL_{L}$ of codimension $h^{2,0}_{\ev}$. This is exactly the
``admissible'' condition needed to apply
\cite[Thm. 1.6]{2023arXiv230316179K} (see also the main result of \cite{2022arXiv221110592E}), which therefore
establishes the existence of components of the desired
codimension. In \emph{op. cit.}, the authors actually prove that 
A3. is equivalent to (2).

After this, the (analytic) density of the typical
components and the non-density of atypical ones follow from the
results of \Cref{section2} (that hold true for arbitrary bases $S$,
and VHS $\VV$). Condition A3. is used in (1), only to ensure that 
the atypical components of $\NL_L$ are necessarly of positive dimension.
\end{proof}

\begin{rmk}
The same proof also shows that, under the weaker condition $\dim U_{Y,L}\geq h^{2,0}_{\ev}$, one still obtains that
the components of $\NL_L$ of codimension $< h^{2,0}_{\ev}$ are not Zariski dense in $U_{Y,L}(\C)$, and the ones of codimension
$h^{2,0}_{\ev}$ are analytically dense.
\end{rmk}
\bibliographystyle{abbrv}
\bibliography{biblio.bib}

\Addresses
\end{document}